\documentclass[10pt,a4paper,reqno]{amsart}
\usepackage{amsfonts,amsmath,amsthm,rotating,lscape,appendix,cite}
\usepackage{graphicx}
\usepackage[left=1in, right=1in, top=1in, bottom=1in, includefoot, headheight=13.6pt]{geometry}

\newtheorem{Theorem}{Theorem}[section]
\newtheorem{Lemma}[Theorem]{Lemma}

\theoremstyle{definition}

\def\nn{\mathbf{n}}

\def\DD{\mathbf{D}}
\def\zz{\mathbf{z}}
\def\uu{\mathbf{u}}

\newcommand\T{\rule{0pt}{4ex}}
\newcommand\B{\rule[-3ex]{0pt}{0pt}}

\begin{document}

\title[Difference Conservation Law Characteristics]{Characteristics of Conservation Laws for Difference Equations}

\author{Timothy J. Grant}
\address{Department of Mathematics\\ University of Surrey\\ Guildford\\ GU2 7XH\\ UK}
\curraddr{Schlumberger Gould Research\\ High Cross\\ Madingley Road\\ Cambridge\\ CB3 0EL\\ UK}
\curraddr{British Antarctic Survey\\ High Cross\\ Madingley Road\\ Cambridge\\ CB3  0ET\\ UK}
\email{tgrant@slb.com}

\author{ Peter E. Hydon} 
\address{Department of Mathematics\\ University of Surrey\\ Guildford\\ GU2 7XH\\ UK}
\email{p.hydon@surrey.ac.uk}
\dedicatory{Communicated by Evelyne Hubert}
\date{\today}

\begin{abstract}
Each conservation law of a given partial differential equation is determined (up to equivalence) by a function known as the characteristic. This function is used to find conservation laws, to prove equivalence between conservation laws, and to prove the converse of Noether's Theorem.  Transferring these results to difference equations is nontrivial, largely because difference operators are not derivations and do not obey the chain rule for derivatives.  We show how these problems may be resolved and illustrate various uses of the characteristic. In particular, we establish the converse of Noether's Theorem for difference equations, we show (without taking a continuum limit) that the conservation laws in the infinite family generated by Rasin and Schiff are distinct, and we obtain all five-point conservation laws for the potential Lotka--Volterra equation.
\end{abstract}

\maketitle

\section{Introduction}
Current research in symmetry methods owes a tremendous debt to  Peter Olver. In particular, his remarkable text, ``Applications of Lie Groups to Differential Equations," remains pre-eminent after more than a quarter of a century. It is a masterpiece of scholarship that is notable for  the lucidity of its exposition and the precision of its proofs. The first (1986) edition was the first text to describe the conditions under which the converse of Noether's Theorem  \cite{Noether} holds. A cornerstone of this result is the proof that, for any system of partial differential equations (PDEs) in Kovalevskaya form that is locally analytic, there is a bijection between equivalence classes of conservation laws and equivalence classes of characteristics. A simpler proof of this result, due to Alonso \cite{Alonso}, is incorporated in the second edition  \cite{Olver} of Olver's text.

Here is a summary of the main definitions for scalar PDEs\footnote{For simplicity, we restrict attention to scalar equations throughout this paper; the corresponding results for systems are contained in the first author's PhD thesis \cite{Grant11}.}.  For a given PDE, $\Delta=0$, a \textit{conservation law} (CLaw) is a divergence expression that vanishes on solutions of the equation, so that
\begin{equation*}
 \mbox{Div} \mathbf{F}=0 \mbox{ when }\Delta=0.
\end{equation*}
A CLaw is \textit{trivial of the first kind} if $\mathbf{F}$ vanishes on solutions of the PDE; it is \textit{trivial of the second kind} if $\mbox{Div} \mathbf{F}\equiv0$. A CLaw is \textit{trivial} if it is a linear combination of the two kinds of trivial CLaws. Two CLaws are \textit{equivalent} if and only if they differ by a trivial CLaw.  If the PDE is totally nondegenerate (see \cite{Olver})  -- for instance, if it is in Kovalevskaya form -- the CLaw can be integrated by parts to find an equivalent CLaw in \textit{characteristic form}, that is, with
\begin{equation}
 \mbox{Div} \mathbf{\tilde F}=Q\Delta.\label{charcont}
\end{equation}
The multiplier $Q$ is called the \textit{characteristic} of the CLaw.

For example, the KdV equation,
\begin{equation*}
\Delta \equiv u_t+uu_x+u_{xxx}=0,
\end{equation*}
has a CLaw with
\begin{align}
 \mbox{Div} \mathbf{F}=&D_t\left(\frac{1}{3}u^3-u_x^2\right)+D_x\left(\frac{1}{4}u^4+u^2u_{xx}-2u_xu_{xxx}+u_{xx}^2-2u_x^2u\right)\nonumber\\
=&\left(u^2-2u_xD_x\right)\Delta.\label{CLawKdV}
\end{align}
Integration by parts yields the characteristic form of (\ref{CLawKdV}) :
\begin{equation*}
\mbox{Div}\mathbf{\tilde F}=\left(u^2+2u_{xx}\right)\Delta.
\end{equation*}
So $Q=\left(u^2+2u_{xx}\right)$ is the characteristic and (\ref{CLawKdV}) is equivalent to the CLaw
\begin{equation}
 D_t\left(\frac{1}{3}u^3-u_x^2\right)+D_x\left(\frac{1}{4}u^4+u^2u_{xx}+2u_xu_t+u_{xx}^2\right)=0.
\end{equation}

A characteristic is said to be \textit{trivial} if it vanishes on solutions of the PDE. By definition, the set of characteristics is a vector space; two characteristics are equivalent if they differ by a trivial characteristic. Therefore, the correspondence between equivalence classes of characteristics and CLaws makes it easy to identify when two seemingly different CLaws are equivalent: one only needs to compare their characteristics. Given a nontrivial characteristic, it is usually easy to reconstruct a corresponding CLaw by inspection; this can also be achieved systematically with the aid of a homotopy operator. In particular, where  Noether's Theorem applies, each characteristic that arises from a one-parameter (local) Lie group of variational symmetries can be used to construct an associated CLaw. 

Until now, Alonso's result has not been transferred to difference equations. Yet one might wish to approximate a given PDE by a finite difference scheme that preserves difference analogues of several CLaws, particularly those that have a clear physical interpretation\footnote{This is one of the oldest branches of geometric integration but, by exploiting the growing power of computer algebra systems, some new strategies for doing this have been developed recently  \cite{Grant11}.}.  This raises the question: is there a function that characterizes each equivalence class of \textit{difference} CLaws? 

Of course, a given difference equation may be interesting in its own right, whether or not it is an approximation to a differential equation. Recent work has shown that CLaws of partial difference equations (P$\Delta$Es) have many features in common with CLaws of PDEs. For instance, Dorodnitsyn \cite{Dorodnitsyn1993,Dorodnitsyn2001} has formulated a finite difference analogue of Noether's theorem. Hydon \& Mansfield \cite{HyM11} studied variational problems whose symmetry generators constitute an infinite-dimensional Lie algebra and derived a difference analogue of Noether's Second Theorem. CLaws of  a given P$\Delta$E can be found directly (whether or not the P$\Delta$E is an Euler--Lagrange equation); see \cite{HydonCLawsPDiffE} for an algorithmic approach that works for any P$\Delta$E in Kovalevskaya form (see below) and see \cite{Rasin, Rasin2006, Rasin2007} for applications of this approach to integrable quad-graph equations. The main shortcoming of the direct construction method is that the algebraic complexity of the computations grows exponentially with the order of the CLaw; in practice, therefore, the method is restricted to low-order CLaws.

Given a difference equation, $\Delta=0$, one cannot obtain the characteristics of its CLaw in the same way one does as for a PDE.  For PDEs, the chain rule ensures that each CLaw is linear in the highest-order derivatives through which the dependence on $\Delta$ occurs.  Integration by parts is then used to find the characteristic.  The analogue of integration by parts for difference equations is summation by parts.  However there is no analogue of the chain rule and therefore CLaws typically depend nonlinearly on $\Delta$ and its shifts. Consequently, it is not possible to construct a characteristic merely by summing by parts. In the current paper, we show how this difficulty can be surmounted.

We also derive and use a difference analogue of Alonso's result. By pulling the characteristic back to a specified set of initial conditions, one can determine a function (the root) which labels the distinct equivalence classes of conservation laws. We show how the root is calculated in practice; examples include an integrable P$\Delta$E with infinitely many CLaws. It is well-known that integrable PDEs have infinite hierarchies of CLaws which can be found using recursion operators, mastersymmetries, or Gardner transformations. Recently, Mikhailov and co-workers \cite{XenitidisCLaws10,XenitidisCoSym10} and Rasin and co-workers \cite{Rasin2010,RasinSchiff09} have shown that the same is true for integrable quad-graph equations. Having used the Gardner transformation to construct such a hierarchy, Rasin \& Schiff \cite{RasinSchiff09} took a continuum limit in order to show that these CLaws are distinct. By using the difference analogue of Alonso's result, we show how to determine directly when CLaws are distinct, irrespective of whether they are preserved in any continuum limit or whether the underlying P$\Delta$E is integrable.

To establish the necessary results, it is helpful to begin by looking at scalar O$\Delta$Es (\S\ref{secOdiffE}).  We define a characteristic of a first integral and show that the characteristic is trivial if and only if the first integral is trivial.  In \S \ref{secPdiffE}, the definition of a characteristic is extended to CLaws of P$\Delta$Es; we prove that there is a bijection between equivalence classes of CLaws and characteristics. This result has several immediate applications. In \S\ref{secGardner}, we use the characteristic to show that the CLaws in the infinite hierarchy for dpKdV generated by the Gardner transformation in \cite{RasinSchiff09} are distinct.  Perhaps the most fundamental application is the establishment of the converse of Noether's Theorem (\S\ref{secNoet}). Consequently, for Euler--Lagrange equations in Kovalevskaya form, there is a bijection between equivalence classes of variational symmetries and CLaws.  Finally, we show how to use the characteristic to find CLaws of a given P$\Delta$E (the potential Lotka-Volterra equation); this provides an alternative to the direct method.

\section{Scalar ordinary difference equations}\label{secOdiffE}

Although the main focus of this paper is on P$\Delta$Es, it is instructive to look at scalar O$\Delta$Es first. The independent variable is  $n \in \mathbb{Z}$ and the dependent variable is $u\in\mathbb{R}$. It is convenient to regard $n$ as a free variable and to denote the shifts of $u$ from a fixed but unspecified $n$ by $u_i:=u(n+i)$. In order to evaluate first integrals on solutions of the O$\Delta$E, one must be able to eliminate the highest (or lowest) shift of $u$. Therefore, we restrict attention to \textit{explicit} $K^{\text{th}}$-order O$\Delta$Es, which are of the form
\begin{equation}
 \Delta:= u_K-\gamma(n,u_0,\hdots,u_{K-1})=0, \qquad \frac{\partial \gamma}{\partial u_{0}} \neq 0, \label{ODiffE2}.
\end{equation}
The set of \textit{initial conditions} is the set of values $\mathbf{z}=\{ n,u_0,\hdots,u_{K-1}\}$ from which all  $u_i$,  $i\ge K$, can be calculated.

A first integral of (\ref{ODiffE2}) is a non-constant function, $\phi(n,u_0,\hdots,u_{K-1})$, that is constant on solutions. It is helpful to introduce the forward shift operator, $S_n$, and the identity operator, $I$, which are defined by
\[
S_n:(n,f(n),u_i)\mapsto (n+1,f(n+1),u_{i+1}),\qquad I:(n,f(n),u_i)\mapsto (n,f(n),u_i);
\]
here $f$ is any function that is defined at $n$ and $n+1$.
In terms of these operators, $\phi$ is constant on solutions if and only if the following difference CLaw holds:
\begin{equation}
 (S_n-I)\phi=0 \qquad \mbox{when } \Delta=0.\label{trivint}
\end{equation}
 It is useful to refer to constant solutions of \eqref{trivint} as trivial first integrals, by analogy with trivial CLaws of the second kind. (Trivial CLaws of the first kind cannot occur when $\phi$ depends only on $n,u_0,\hdots,u_{K-1}$.)

A nontrivial first integral must depend on $u_{K-1}$, otherwise  $(S_n-I)\phi$ does not depend on $\Delta$ (in which case, the only way for $\phi$ to be a first integral is to be identically constant). Therefore, the CLaw can be written as
\begin{equation}
 C(\mathbf{z},\Delta):=\phi(n\!+\!1,u_1,\hdots,u_{K-1},\Delta+\gamma(n,u_0,\hdots,u_{K-1}))-\phi(n,u_0,\hdots,u_{K-1}),\label{ODECLaw}
\end{equation}
where $C(\mathbf{z},0)=0$.  We now use the Fundamental Theorem of Calculus to write the CLaw in the form
\begin{align*}
 C(\mathbf{z},\Delta)=&\int_{\lambda=0}^1 \frac{d}{d\lambda} C(\mathbf{z},\lambda\Delta) \, d\lambda\\
=&\Delta \int_{\lambda=0}^1 \phi_{,K}(n+1,u_1,\hdots,u_{K-1},\lambda\Delta+\gamma(n,u_0,\hdots,u_{K-1})) \, d\lambda.
\end{align*}
(Throughout this paper, the partial derivative of a function, $f$, with respect to its $i^{\text{th}}$ \textit{continuous} argument is denoted by $f_{,i}$). 
By analogy with differential equations,  we define the characteristic to be the multiplier
\begin{equation}
 Q(\mathbf{z},\Delta):=\int_{\lambda=0}^1 \frac{\partial C(\mathbf{z},\lambda\Delta)}{\partial \lambda \Delta} \, d\lambda
=\frac{(S_n\phi)|_{u_K=\Delta+\gamma} - (S_n\phi)|_{u_K=\gamma} }{\Delta }\,.\label{ODELTAQ}
\end{equation}
As with  differential equations, a trivial characteristic is one that vanishes on solutions, so that $Q(\mathbf{z},0)=0$.

A trivial first integral is a constant, so $C(\mathbf{z},\Delta)\equiv0$; therefore any trivial first integral has a trivial characteristic.
To show that any trivial characteristic corresponds to a trivial first integral, it is helpful to define the \textit{root} of the characteristic to be the function
\begin{align}
 \overline{Q} (\zz) := \lim_{\mu \to 0 } Q(\mathbf{z},\mu\Delta) 
=&   \lim_{\mu \to 0} \frac{(S_n\phi)|_{u_K=\mu\Delta+\gamma} - (S_n\phi)|_{u_K=\gamma} }{\mu\Delta }  \nonumber \\ 
=&\,\phi_{,K}(n+1,u_1,\hdots,u_{K-1},\gamma(n,u_0,\hdots,u_{K-1}));\label{Qzero}
\end{align}
to do this, we require that $S_n \phi$ is differentiable in its $K^{\text{th}}$ continuous argument at $u_K=\gamma$.
If the characteristic is trivial, its root is zero, so
\begin{equation}
 0=\phi_{,K}(n+1,u_{1},\hdots,u_{K-1},\gamma(n,u_{0},\hdots,u_{K-1})).\label{eqn1}
\end{equation}
As $\gamma_{,1} \neq 0$, there is only one way for (\ref{eqn1}) to be satisfied identically in $u_0$: the CLaw \eqref{ODECLaw} cannot depend on $\Delta$, so $\phi$ must be a trivial first integral.

Having dealt with the question of triviality, we now show how to reconstruct the first integral from the root.  For clarity, we begin with a second-order example, but the same procedure applies in general.
The O$\Delta$E
\begin{equation}
 \Delta:= u_{2}-\gamma(n,u_{0},u_{1})=0, \qquad \gamma(n,u_{0},u_{1})=\frac{1}{4}\left(u_1+5u_0+3\left\{1+(u_1+u_0)^2\right\}^{1/2}\right),\label{exampleOdiffE}
\end{equation}
has a first integral
\begin{equation}
 \phi(n,u_0,u_1)=2^{-n}\left(u_1+u_0+\left\{1+(u_1+u_0)^2\right\}^{1/2}\right).\label{firstintode}
\end{equation}
 Consequently, the characteristic is
\[
Q(\mathbf{z},\Delta)=2^{-(n+1)}\left(\Delta+\left\{1+(\Delta+\gamma(n,u_{0},u_{1})+u_1)^2\right\}^{1/2}-\left\{1+(\gamma(n,u_{0},u_{1})+u_1)^2\right\}^{1/2}\right)/\Delta,
\]
and so the root is
\begin{align}
 \overline{Q} (\zz) &=2^{-(n+1)}\left(1+\big\{\gamma(n,u_{0},u_{1})+u_1\big\}\left\{1+(\gamma(n,u_{0},u_{1})+u_1)^2\right\}^{-1/2}\right)\label{ex1q1}\\
&=\frac{2^{2-n}\left(5+2(u_1+u_0)^2+2(u_1+u_0)\left\{1+(u_1+u_0)^2\right\}^{1/2}\right)}{25+16(u_1+u_0)^2}\,.\label{ex1q2}
\end{align}

To reconstruct the first integral from the root, one must reverse this process. First use the  O$\Delta$E (\ref{exampleOdiffE}) to eliminate $u_0$ from (\ref{ex1q2}), obtaining (\ref{ex1q1}). Now treat $u_1$ and $u_2$ as the continuous variables in (\ref{Qzero}), which amounts to
\begin{equation*}
 \frac{\partial}{\partial u_2}\phi(n+1,u_1,u_2)= 2^{-(n+1)}\left(1+\big(u_2+u_1\big)\left\{1+(u_2+u_1)^2\right\}^{-1/2}\right).
\end{equation*}
Solving this and  then applying $S_n^{-1}$ gives
\[
 \phi(n,u_0,u_1)=2^{-n}\left(u_1+\left\{1+(u_1+u_0)^2\right\}^{1/2}\right)+f(n,u_0),
\]
where $f(n,u_0)$ is yet to be determined. The determining equation is (\ref{trivint}), which amounts (after simplification) to
\begin{equation}
f(n+1,u_1)-f(n,u_0)=2^{-(n+1)}(u_1-2u_0).\label{ode3}
\end{equation}
Differentiating this with respect to $u_0$ gives
\[
f_{,1}(n,u_0)=2^{-n},
\]
and so
\[
f(n,u_0)=2^{-n}u_0 + g(n).
\]
Thus, (\ref{ode3}) yields the O$\Delta$E
\[
g(n+1)-g(n)=0.
\]
Consequently, $g(n)$ is an (irrelevant) arbitrary constant, which can be set to zero without loss of generality. This completes the reconstruction of the first integral (\ref{firstintode}).

The process of reconstructing the first integral for a general scalar O$\Delta$Es is similar. 
For any first integral, $\phi$, of the $K^{\text{th}}$-order O$\Delta$E \eqref{ODiffE2},
the root is
\begin{equation}
 \overline{Q}(\zz) := Q(\zz,0)=\phi_{,K}(n+1,u_1,\hdots,u_{K-1},\gamma(\zz)).\label{eq100}
\end{equation}
Given a root, eliminate $u_0$ in  favour of $u_K$ to obtain
\begin{equation*}
\overline{Q}(n,u_0(n,u_1,\hdots,u_K),u_1,\hdots,u_{K-1})=\frac{\partial}{\partial  u_K } \phi(n+1,u_1,\hdots,u_{K-1},u_K).
\end{equation*}
Integrating this and applying $S_n^{-1}$ yields
\begin{equation*}
 \phi(n,u_0,\hdots,u_{K-1})=\int \overline{Q}(n-1,u_0,\hdots,u_{K-1}) d u_{K-1} + f(n,u_0,\hdots,u_{K-2}).
\end{equation*}
All that remains is to find $f$.  Substitute $\phi$ into \eqref{trivint} and simplify to obtain the determining equation for $f$. Differentiate this with respect to $u_0$, then integrate to obtain $f$ up to an arbitrary function $g(n,u_1,\hdots,u_{K-2})$. Obtain the determining equation for $g$, apply $S_n^{-1}$, and repeat the whole process with $S_n^{-1}g$ replacing $f$. Continue in the way until the remaining function to be determined depends on $n$ only.  The determining equation for this function (which we call $h$) is of the form
\[
h(n+1)-h(n)= H(n),
\]
 where $H(n)$ is given. The solution is obtained by summation; this completes the reconstruction of the CLaw.

\section{Partial Difference Equations}\label{secPdiffE}

We now generalize the ideas from the last section to scalar P$\Delta$Es for $u\in\mathbb{R}$ with two independent variables\footnote{The corresponding results for systems of difference equations with arbitrarily many independent variables are obtained \textit{mutatis mutandis}; see \cite{Grant11}.} $\mathbf{n}=(m,n)\in \mathbb{Z}^2$. Again, we regard the independent variables as being free; given $(m,n)$, let $u_{ij}:=u(m+i,n+j)$. [The indices $i$ and $j$ may be negative: each minus sign in the subscript should be treated as being attached to the following digit. For instance, $u_{1 -1}$ denotes $u(m+1,n-1)$.] The action of the shift and identity operators on $(m,n)$ induces an action on every function $f(m,n)$, and in particular on $u_{ij}$, as follows:
\begin{align*}
 S_m:(m,n,f(m,n),u_{ij}) &\mapsto (m+1,n,f(m+1,n),u_{(i+1)j}), \\
 S_n:(m,n,f(m,n),u_{ij}) &\mapsto (m,n+1,f(m,n+1),u_{i(j+1)}), \\
 I:(m,n,f(m,n),u_{ij})&\mapsto (m,n,f(m,n),u_{ij}).
\end{align*}
A  P$\Delta$E is written as
\begin{equation}\label{pdiffe}
 \Delta(m,n,[u])=0,
\end{equation}
where $[\cdot]$ denotes  the argument $\cdot$ and a finite number of its shifts.

A CLaw of a P$\Delta$E is a divergence expression\footnote{For differential equations on $\mathbb{R}^N$ and difference equations on $\mathbb{Z}^N$, the set of divergence expressions is the kernel of the Euler--Lagrange operator \cite{Kupershmidt}.} that vanishes on solutions of the PDE: 
\begin{equation}
\mbox{Div} {\mathbf{F}} := (S_m-I)F+(S_n-I)G=0 \mbox{ when } [{\Delta}]=\mathbf{0}.\label{CLawdef}
\end{equation}
The functions $F:=F(m,n,[u])$ and $G:=G(m,n,[u])$ are the densities of the CLaw.
In the same way as for PDEs, a CLaw of a P$\Delta$E is trivial if and only if it is a linear combination of the following two kinds of trivial CLaws.
\begin{enumerate}
 \item[] \textbf{First kind}: $\mathbf{F}|_{[{\Delta}]=\mathbf{0}}=\mathbf{0}$. The densities vanish on solutions, so we call these \textit{trivial densities}.
\item[] \textbf{Second kind}: $\mbox{Div} \mathbf{F} \equiv 0$, without reference to the equation ${\Delta}=0$ and its shifts. This occurs if there exists a function $H$ such that $F=(S_n-I)H$ and $G=-(S_m-I)H$.
\end{enumerate}

\subsection{Kovalevskaya Form}\label{Koveysec}
A crucial step in dealing with first integrals is to replace $u_K$ by $\Delta+\gamma$.  We can do something similar for a $K$th-order P$\Delta$E (\ref{pdiffe}) if it is in \textit{Kovalevskaya form}\footnote{For a given PDE in Kovalevskaya form, any equation that holds on solutions of the PDE can be pulled back to an identity on the initial conditions; with the above definition, the same is true for P$\Delta$Es.},
\begin{equation}
  \Delta :=u_{K0}-\omega(m,n,\mathbf{u_{0}},\mathbf{u_{1}},\hdots,\mathbf{u_{K-1}})=0,\label{scalarkovey}
\end{equation}
where  $\mathbf{u_{i}}=\{u_{ij}:\ j\in\mathbb{Z}\}$ and there exists $j$ such that $\partial \omega/\partial u_{0j}\neq 0$.
 A schematic example of a 2-D scalar P$\Delta$E in Kovalevskaya form is shown in  Figure \ref{fig1}. Let $\mathbf{z}=\{m,n, \mathbf{u_{0}},\mathbf{u_{1}},\hdots,\mathbf{u_{K-1}} \}$ be the (minimal) initial conditions from which shifts of \eqref{scalarkovey} can be used to find any point of the form $u_{lj}$ for $l \ge K$.  The function $\omega:=\omega(\mathbf{z})$ depends on a finite subset of these points. It is convenient to denote shifts of $\omega$ by $\omega_{ij}:= S_m^iS_n^j\omega$.
\begin{center}\begin{figure}[tb]\begin{center}
 \includegraphics[trim = 0mm 0mm 0mm 0mm, clip,scale=1]{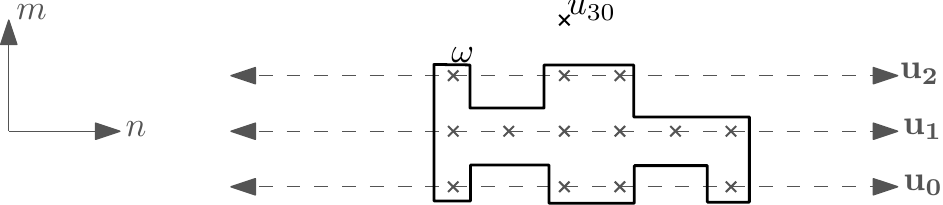}
\caption[A 2-dimensional P$\Delta$E in Kovalevskaya form]{A P$\Delta$E in Kovalevskaya form: $u_{30}=\omega(m,n,\mathbf{u_{0}},\mathbf{u_{1}},\mathbf{u_{2}})$. The box encloses the values $u_{ij}$ on which $\omega$ depends; these are represented by crosses.  The dashed lines represent the initial conditions that would be required to obtain all $u_{ij}$ in the upper half-plane.}\label{fig1}\end{center}
\end{figure}\end{center}

A scalar P$\Delta$E with two independent variables is \textit{explicit} if it can be transformed into Kovalevskaya form by an admissible change of independent variables, that is,  by a bijective linear map from $\mathbb{Z}^2$ to itself. The new independent variables are
\begin{equation}
 \left(\begin{array}{c} \tilde m\\ \tilde n \end{array}\right)= A\left(\begin{array}{c} m\\ n \end{array}\right)+\mathbf{b} \quad \mbox{where } A\in GL(2,\mathbb{Z}), \mbox{ det}(A)=\pm1,\mbox{ and } \mathbf{b} \in \mathbb{Z}^2. \label{explicittrans}
\end{equation}
Although the value of the $u$ at each point is unchanged, the coordinates of the point have changed, so it is helpful to define
\begin{equation}\label{cov}
 \tilde u({\tilde m, \tilde n}):=u(m(\tilde m , \tilde n), n(\tilde m , \tilde n) ).
\end{equation}
Now fix $(m,n)$; using the shorthand $u_{ij}=u(m+i,n+j)$ and setting $\tilde{u}_{00}=u_{00}$, we obtain
\[
u_{ij}=S_m^iS_n^ju_{00}=\tilde{u}(\tilde{m}(m+i,n+j),\tilde{n}(m+i,n+j))
=\tilde{u}_{\{\tilde{m}(m+i,n+j)-\tilde{m}(m,n)\}\{\tilde{n}(m+i,n+j)-\tilde{n}(m,n)\}}.
\]
For instance, the shear 
\begin{align}
 A=\left( \begin{array}{cc} 1 & 1 \\ 0 & 1 \end{array} \right), \quad \mathbf{b}=\left(\begin{array}{c} 0\\ 0 \end{array} \right) \label{quadshear}
\end{align}
transforms any quad-graph equation, 
\begin{equation*}
 u_{11}=\omega(m,n,u_{00},u_{10},u_{01}),
\end{equation*}
into the Kovalevskaya form
\[
\tilde{u}_{21}=\omega(m(\tilde{m},\tilde{n}),n(\tilde{m},\tilde{n}),\tilde{u}_{00},\tilde{u}_{10},\tilde{u}_{11}).
\]
In particular, the dpKdV equation (\textbf{H1} in the ABS classification \cite{ABS}),
\begin{equation}
 u_{11}=u_{00}+\frac{\beta-\alpha}{u_{10}-u_{01}}\,,\label{H1}
\end{equation}
is transformed into
\begin{equation}
 \tilde u _{21}= \tilde u_{00}+\frac{\beta - \alpha}{ \tilde u_{10} - \tilde u_{11}}\,.\label{dKdVtrans}
\end{equation}
\begin{figure}[tb]
\begin{center}
 \includegraphics[trim = 0mm 0mm 0mm 0mm, clip,scale=0.9]{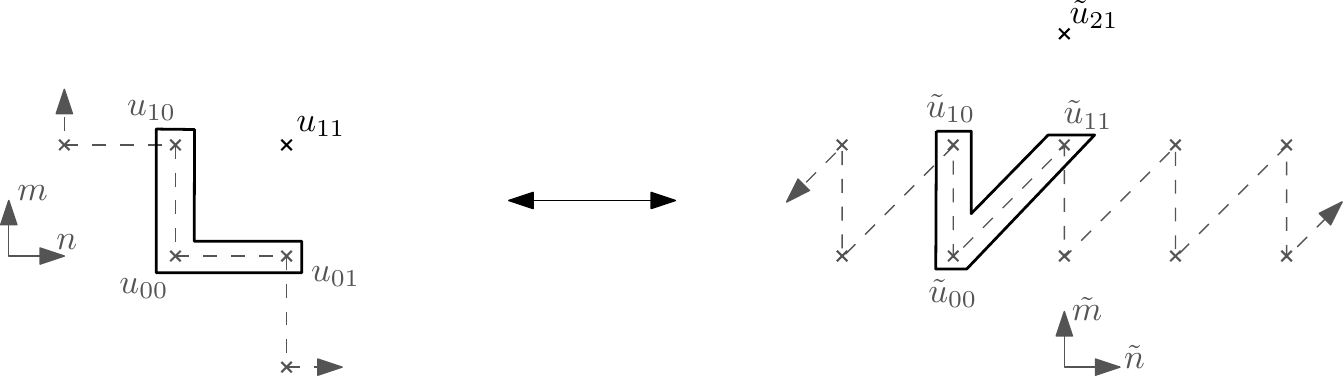}
\caption[Transformation of a quad-graph equation into Kovalevskaya form]{Transformation of a quad-graph equation into Kovalevskaya form. The dashed lines show the initial conditions that would be required to determine all $u_{ij}$ (respectively $\tilde{u}_{ij}$) in the upper-right (respectively upper) half-plane.}\label{fig1b}
\end{center}
\end{figure}

\begin{Lemma}\label{TransLem}
 When an explicit scalar P$\Delta$E is transformed according to \eqref{explicittrans} and \eqref{cov}, there is a bijective correspondence between equivalence classes of CLaws of the original P$\Delta$E and the transformed P$\Delta$E.
\end{Lemma}
\begin{proof}
The transformation  is of the form
\begin{align*}
 \left( \begin{array}{c} \tilde m \\ \tilde n \end{array} \right) = \left( \begin{array}{cc} a & b \\ c& d \end{array} \right) \left( \begin{array}{c} m \\ n \end{array} \right) + \left( \begin{array}{c} e \\ f \end{array}\right),\qquad ad-bc=\pm 1,\quad a,b,c,d,e,f\in\mathbb{Z},
\end{align*}
so the effect of the original shift operators on the transformed independent variables is
\begin{align*}
 S_m \tilde m = a(m+1)+bn+e = \tilde m + a, \quad S_m \tilde n = c(m+1)+dn + f = \tilde n + c,\\
S_n \tilde m = am+b(n+1)+e = \tilde m + b, \quad S_n \tilde n = cm+d(n+1) + f = \tilde n + d.
\end{align*}
Thus
\begin{equation*}
 S_m= S_{\tilde m}^a  S_{\tilde n}^c, \quad S_n= S_{\tilde m}^b  S_{\tilde n}^d.
\end{equation*}
Therefore given a CLaw with densities $F$ and $G$,
\begin{align}
(S_m-I)F+(S_n-I)G &= ( S_{\tilde m}^a  S_{\tilde n}^c-I)\tilde F + ( S_{\tilde m}^b  S_{\tilde n}^d-I)\tilde G\nonumber\\
&=\left\{( S_{\tilde m}^a -I)\tilde F+ ( S_{\tilde m}^b -I)\tilde G\right\} +\left\{ ( S_{\tilde n}^c -I)S_{\tilde m}^a\tilde F+ ( S_{\tilde n}^d -I)S_{\tilde m}^b\tilde G\right\}\nonumber\\
&= ( S_{\tilde m} -I)\hat F+ ( S_{\tilde n}-I)\hat G,\label{CLawsTrans}
\end{align}
where $\tilde F$ and $\tilde G$ denote the densities in terms of the transformed variables and the last equality is obtained by factorizing each expression in braces.
If the original CLaw is trivial of the second kind then it vanishes identically, so the transformed CLaw must also vanish identically.  If the original densities $F$ and $G$ are trivial then so are $\tilde F$ and $\tilde G$;  consequently, the new densities $\hat F$ and $\hat G$ will also be trivial.  As the transformation is invertible, the converse is also true, which completes the proof. 
\end{proof}

For example, when a quad-graph equation is transformed by
 \eqref{quadshear}, the transformed equation has CLaws with the densities 
\begin{equation}
 \hat F = \tilde F + \tilde G, \quad \hat G =  S_{\tilde m} \tilde G.\label{H1transCLaw}
\end{equation}
To transform back to the quad-graph, use
\begin{align}
 \left( \begin{array}{c}  m \\  n \end{array} \right) = \left( \begin{array}{cc} 1 & -1 \\ 0& 1 \end{array} \right) \left( \begin{array}{c} \tilde m \\ \tilde n \end{array} \right).\label{quadreturn}
\end{align}
Then the densities for the CLaws are
\begin{equation}
 F=\hat F - S_m^{-1}\hat G, \quad  G = S_m^{-1}\hat G.\label{H1transback}
\end{equation}
Generally speaking, the new densities will depend on variables other than the transformed initial conditions; however, the difference equation can be used to pull them back to expressions that depend only on the transformed initial conditions.

For later use, the dpKdV equation has four three-point and three five-point CLaws \cite{Rasin2007}; the transformed equation \eqref{dKdVtrans} has corresponding CLaws whose densities are listed in Table \ref{table1b} (where the tildes have been dropped to prevent clutter).
\begin{footnotesize}
\begin{table}[tp]
\begin{center}
\caption{Densities for the three and five point CLaws of transformed dpKdV}
\begin{tabular}{l}
 \hline
$ \hat  F_1=\left( -1 \right) ^{m} \left( 2\,u_{00}(u_{11}-u_{10})+\alpha-\beta \right)$\T\\$
\hat G_1= \left( -1 \right) ^{m} \left( 2\,u_{10} \left( u_{0-1}+{
\frac {\beta-\alpha}{u_{1-1}-u_{10}}} \right) -\alpha
 \right) \label{CLaw1} $ \B\\ \hline 
$ \hat F_2= \left( u_{00}-u_{10} \right)  \left( u_{00}u_{10}-\alpha \right)- \left( u_{00}-u_{11} \right)  \left( u_{00}u_{11}-\beta \right)  $
 \T\\ $
\hat G_2= \left( u_{10}-u_{0-1}-{\frac {\beta-\alpha}{u_{1-1}-u_{10}}} \right)  \left( u_{10} \left( u_{0-1}+{\frac {
\beta-\alpha}{u_{1-1}-u_{10}}} \right) -\alpha \right)\label{CLaw2} $
 \B\\ \hline $
\hat F_3= \left( -1 \right) ^{m} \big( {u_{00}}(u_{11}-u_{10})(u_{00}+u_{10}+u_{11})
+\alpha(u_{00}+u_{10})-\beta(u_{00}+u_{11}) \big) $
 \T\\$
\hat G_3=  \left( -1 \right) ^{m} \left( u_{10}+u_{0-1}+{\frac {\beta
-\alpha}{u_{1-1}-u_{10}}} \right)  \left( u_{10}
 \left( u_{0-1}+{\frac {\beta-\alpha}{u_{1-1}-u_{10}}}\label{CLaw3}
 \right) -\alpha \right) $
\B\\ \hline $
\hat F_4= \left( -1 \right) ^{m+1} \left(2\,{u_{00}}^{2}({u_{10}}^{2} -{u_{11}}^{2})+4\,u_{00}(\beta u_{11}-\alpha u_{10})+{\alpha}^{2}-{\beta}^{2} \right) $
\T \\$
\hat G_4=  \left( -1 \right) ^{m} \left( 2\,{u_{10}}^{2} \left( u_{0-1}+{\frac {\beta-\alpha}{u_{1-1}-u_{10}}} \right) ^{2}-4\,
\alpha\,u_{10} \left( u_{0-1}+{\frac {\beta-\alpha}{u_{1-1}-u_{10}}} \right) +{\alpha}^{2} \right)\label{CLaw4}
$\B\\ \hline
$\hat F_5= -\ln  \left( {\frac {\beta-\alpha}{u_{10}-u_{11}}} \right) +\ln 
 \left( u_{0-1}-u_{00}+{\frac {\beta-\alpha}{u_{1-1}-u_{10}}} \right) 
 $
\T\\ 
$\hat G_5= \ln  \left( u_{1-1}-u_{10}+ \left( \beta-\alpha \right)  \left( u_{0-2}-u_{0-1}
+{\frac {\beta-\alpha}{u_{1-2}-u_{1-1}}}-{\frac {
\beta-\alpha}{u_{1-1}-u_{10}}} \right) ^{-1} \right) 
$\B\\ \hline
$\hat F_6 =-\ln  \left( u_{00}-u_{0-1}+{\frac {\beta-\alpha}{u_{10}-u_{11}}} \right) +\ln  \left( {\frac {\beta-\alpha}{u_{1-1}-u_{10}}} \right)
 $\T\\
$\hat G_6 = \ln  \left(  \left( \beta-\alpha \right)  \left( u_{0-2}-u_{0-1}+{\frac {
\beta-\alpha}{u_{1-2}-u_{1-1}}}-{\frac {\beta-\alpha}
{u_{1-1}-u_{10}}} \right) ^{-1} \right) \label{CLaw6}
 $\B\\ \hline
$\hat F_7= (m-n)\hat F_5 + n \hat F_6 $\T\\ 
$\hat G_7=(m-n+1)\hat G_5 + n \hat G_6 $\B\\ \hline 
\end{tabular}\label{table1b}\end{center}\end{table}
\end{footnotesize}

\subsection{The characteristic}\label{defChar}
The characteristic of a given CLaw of a P$\Delta$E in Kovalevskaya form (\ref{scalarkovey}) is defined as follows.  Any $u_{ij}$ may be written in terms of the initial conditions, $\mathbf{z}$, and $[{\Delta}]$. Given a differentiable function, $C(\mathbf{z},[{\Delta}])$, that satisfies $C(\mathbf{z},[{0}])=0$, the Fundamental Theorem of Calculus yields 
 \[
  C(\mathbf{z},[{\Delta}])=\int_{\lambda=0}^{1} \frac{d}{d\lambda} C(\mathbf{z},[\lambda \Delta]) \,d\lambda
=\int_{\lambda=0}^1 \sum_{i,j} (S_m^i S_n^j \Delta) \frac{\partial C(\mathbf{z},[\lambda \Delta])}{\partial( S_m^iS_n^j\lambda\Delta)} \,d\lambda.
 \]
Summing by parts gives
\begin{align}
 C(\mathbf{z},[\Delta])=&\Delta \int_{\lambda=0}^1 \left\{E_{\Delta}\left(C(\mathbf{z},[{\Delta}])\right)\right\}\!\big|_{\Delta \mapsto \lambda \Delta} \, d\lambda + \mbox{Div} {\mathbf{F}},\label{eq1111}
\end{align}
where $E_\Delta$ is the difference Euler--Lagrange operator that corresponds to variations in $\Delta$,
\begin{equation}
 E_{\Delta}(C(\mathbf{z},[\Delta])):= \sum_{i,j} S_m^{-i}S_n^{-j} \frac{\partial C(\mathbf{z},[\Delta])}{\partial (S_m^iS_n^j \Delta)}\,,
\end{equation}
and where $\mathbf{F}$ has two components, each of which vanishes on solutions of the P$\Delta$E. (The generalization to P$\Delta$Es with  more than two independent variables is obvious.)
If, in addition, $C(\mathbf{z},[{\Delta]})$ is a divergence expression (so that it is a CLaw) then, by subtracting the trivial CLaw $\mbox{Div} {\mathbf{F}}$ from both sides of (\ref{eq1111}), one obtains the equivalent CLaw 
\begin{equation*}
 \tilde{C}(\mathbf{z},[{\Delta}])= C(\mathbf{z},[{\Delta}])-\mbox{Div} {\mathbf{F}}= {Q}(\mathbf{z},[{\Delta}]) \cdot {\Delta} \end{equation*}
where  
\begin{equation}
 Q(\mathbf{z},[{\Delta}]):=\int_{\lambda=0}^1 \left\{\left(E_{\Delta}(C(\mathbf{z},[{\Delta}]))\right)\right\}\big|_{\Delta \mapsto \lambda \Delta} \, d\lambda.\label{char1}
\end{equation}
The function ${Q}$ is a multiplier of the difference equation, so to be consistent with the continuous case and \cite{MarunoQuispel06}, it is defined to be the characteristic of the CLaw.

Just as for O$\Delta$Es, a trivial characteristic is one that vanishes on solutions, so ${Q}(\mathbf{z},[{0}])={0}$. 
To factor out trivial characteristics, we define the \textit{root} of $Q$ to be
\begin{equation}
 \overline{Q}(\zz):= \lim_{\mu \to 0} Q(\zz,[\mu\Delta]) = \lim_{\mu \rightarrow 0} \int_{\lambda=0}^1 \left\{E_{\mu\Delta}(C(\mathbf{z},[\mu{\Delta}]))\right\}\big|_{\Delta \mapsto \lambda \Delta} \, d\lambda
=\left\{E_{\Delta}(C(\mathbf{z},[\Delta]))\right\}\big|_{[{\Delta}]=\mathbf{0}}.\label{Qsols}
\end{equation}
 For PDEs, the characteristic of a CLaw is trivial if and only if the CLaw is trivial.  We will now show that, with our definition of the characteristic (\ref{char1}), the same is true for P$\Delta$Es in Kovalevskaya form.
Hence, we need to prove that a CLaw is trivial if and only if $\overline{Q}(\zz)=0$.

\subsection{A trivial CLaw implies a trivial characteristic}\label{sectrivial}
A trivial CLaw of the second kind ($\mbox{Div} \mathbf{F}\equiv 0$) vanishes identically, so the proof is immediate.  
However the first kind of triviality takes the form $ \mathbf{F}(\mathbf{z},[{0}])=\mathbf{0}$.  To deal with this  we use the identity
\begin{align}
 &E_{\Delta}\left(\mbox{Div} \textbf{F}(\mathbf{z},[\mathbf{\Delta}])\right)=E_{\Delta}\Bigl((S_m-I) \bigl({F}(\mathbf{z},[{\Delta}])-{F}(\mathbf{z},[{0}])\bigr)+(S_n-I) \bigl({G}(\mathbf{z},[{\Delta}])-{G}(\mathbf{z},[{0}])\bigr)\Bigr)\nonumber\\
&=\sum_{i,j} S_m^{-i}S_n^{-j} \left(\! \sum_{k} \frac{\partial S_m z_k}{\partial (S_m^iS_n^j \Delta)}\cdot S_m\frac{\partial\left\{F(\mathbf{z},[{\Delta}])\!-\!F(\mathbf{z},[0])\right\}}{\partial z_k}   \!\right)\!-\! \sum_{i,j} S_m^{-i}S_n^{-j} \left(  \frac{\partial F(\mathbf{z},[\Delta])}{\partial S_m^iS_n^j \Delta^l}+\frac{\partial G(\mathbf{z},[\Delta])}{\partial (S_m^iS_n^j \Delta)}\right) \nonumber\\
&\quad + \sum_{i,j} S_m^{-i}S_n^{-j} \left( \sum_{l,k} \frac{\partial S_mS_m^lS_n^k \Delta}{\partial S_m^iS_n^j \Delta} \cdot S_m\frac{\partial F(\mathbf{z},[{\Delta}])}{\partial S_m^lS_n^k \Delta} +\frac{\partial S_nS_m^lS_n^k \Delta}{\partial S_m^iS_n^j \Delta} \cdot S_n\frac{\partial G(\mathbf{z},[{\Delta}])}{\partial S_m^lS_n^k \Delta}\right),
 \label{eqn9}
\end{align}
where $z_k \in \zz$ and we have used the fact that (for equations in Kovalevskaya form) $S_n z_k \in \zz$.
The identity
\begin{equation*}
 \sum_{i,j,k,l} S_m^{-i}S_n^{-j} \left( \frac{\partial S_mS_m^lS_n^k \Delta}{\partial S_m^iS_n^j \Delta} \cdot S_m\frac{\partial F(\mathbf{z},[{\Delta}])}{\partial S_m^lS_n^k \Delta}\right)=\sum_{k,l} S_m^{-1}S_m^{-l}S_n^{-k}S_m\frac{\partial F(\mathbf{z},[{\Delta}])}{\partial S_m^lS_n^k \Delta}\,,
\end{equation*}
together with a similar identity for the $G$ term, simplifies (\ref{eqn9}) to
\begin{equation}
 E_{\Delta}(\mbox{Div} \mathbf{F}(\mathbf{z},[{\Delta}]))=\sum_{i,j,k} S_m^{-i}S_n^{-j} \left(\frac{\partial S_m z_k}{\partial S_m^iS_n^j\Delta}\cdot S_m\frac{\partial}{\partial z_k}\left(F(\mathbf{z},[{\Delta}])-F(\mathbf{z},[0])\right)   \right).\nonumber
\end{equation}
From this, it immediately follows that
\begin{equation}
\big\{  E_{\Delta}\left(\mbox{Div} \textbf{F}(\mathbf{z},[{\Delta}])\right)  \big\}\Big|_{[{\Delta}]=\mathbf{0}}=0,\nonumber
\end{equation}
so the root is zero.

\subsection{A trivial characteristic implies a trivial CLaw}\label{seccon}
Having proved that the characteristic of a trivial CLaw is a trivial characteristic, it is now clear that trivial densities may be added to the densities of any CLaw without affecting the root.
In particular, the components of $\textbf{F}(\mathbf{z},[\mathbf{\Delta}])-\textbf{F}(\mathbf{z},[\mathbf{0}])$ are trivial densities,  so any explicit dependence on $[\mathbf{\Delta}]$ in the densities can be removed.  
Therefore, we now assume that the densities of any particular CLaw for (\ref{scalarkovey}) do not depend explicitly on $[\mathbf{\Delta}]$; with this assumption, only the second kind of triviality can occur.  Consequently, the CLaw has densities of the form
\begin{equation*}
 F:=F(m,n,\mathbf{u_0},\hdots,\mathbf{u_{K-1}}), \quad \mbox{and} \quad G:=G(m,n,\mathbf{u_0},\hdots,\mathbf{u_{K-1}}).
\end{equation*}
Clearly, $S_nG$ does not depend on $\mathbf{u_{K}}$, so the root is
\begin{equation*}
 \overline{Q} = \sum_{j=0}^R S_n^{-j}\frac{\partial }{\partial S_n^j \Delta} F(m+1,n,\mathbf{u_1},\hdots,\mathbf{u_{K-1}},[S_n^j \Delta + \omega_{0j}])\big|_{[\Delta]=0}.  
\end{equation*}
Here we have assumed that $F$ depends on $\{u_{K-1,j}, j=0,\dots,R\}$ and on no other $u_{K-1,j}$; there always exists an $R\geq 0$ for which this assumption is valid, as we have the freedom to select the starting-point $(m,n)$ relative to which all $u_{ij}$ are compared.
The shifts of $S_mF$ on which $\overline{Q}$ depends are shown in Figure \ref{fig4}; here and henceforth (for brevity), we refer to the values $u_{ij}$ that occur in any expression as `points' on which the expression depends.

 Let $u_{0L}$ be the leftmost point in $\mathbf{u_{0}}$ on which $\omega=\omega_{00}$ depends.  Then
\begin{equation*}
\frac{\partial \omega_{0j}}{\partial  u_{0(L+j)} } \neq 0, \quad j \in  \mathbb{Z}; \qquad\quad\frac{\partial \omega_{0i}}{\partial u_{0(L+j)} } =0, \quad i > j \quad i,j \in \mathbb{Z}.
\end{equation*}
On solutions of the P$\Delta$E, each point $u_{0(L+j)}$ (such as the square points in Figure \ref{fig4}) may be replaced by $\omega_{0j}$ (represented by the discs in Figure \ref{fig4}) as an  independent variable in $\overline{Q}$, because the determinant of the Jacobian \begin{small}$\frac{\partial (  \omega_{0-\!R}, \hdots,\omega_{0R} ) }{\partial (u_{0(L-R)},\hdots,u_{0(L+R)} ) } $\end{small} is nonzero. 
Therefore
\begin{equation}
  \overline{Q} = \sum_{j=0}^R S_n^{-j}\frac{\partial }{\partial \omega_{0j}} F(m+1,n,\mathbf{u_1},\hdots,\mathbf{u_{K-1}},[\omega_{0j}])=: E_{\omega}(S_mF). \label{euleromega}
\end{equation}
This is a restricted difference Euler operator corresponding to variations in $\omega$; in effect, it treats $m$ and the $\mathbf{u_i}$ terms as parameters.  The kernel of this operator is made up of total divergences of the form $(S_n-I)H$, and functions of $\zz\backslash\uu_\mathbf{0}$ only (see Lemma \ref{EulerLem} below).  Thus if the characteristic is zero on solutions of the P$\Delta$E then
\begin{figure}[t]
\begin{center}
 \includegraphics[trim = 0mm 0mm 0mm 0mm, clip,scale=0.7]{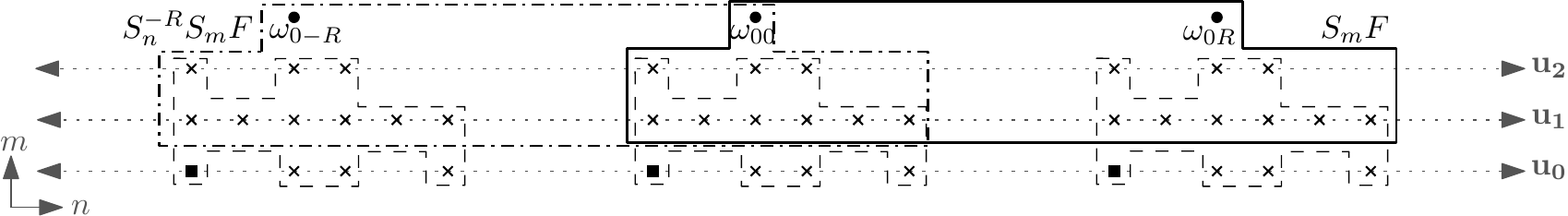}
\caption[A graphical representation of the terms in the characteristic]{A graphical representation of the terms in the characteristic for a scalar P$\Delta$E with two independent variables.  The solid black box encloses points on which $S_mF$ depends and the dash-dot box encloses points on which $S_n^{-R}S_mF$ depends. The dashed boxes enclose points on which $\omega_{0-\!R}$, $\omega_{00}$ and $\omega_{0R}$ depend.}\label{fig4}
\end{center}\end{figure}
\[
F(m\!+\!1,n,\mathbf{u_1},\hdots,\mathbf{u_{K-1}},[\omega_{0j}])=(S_n-I)H(m\!+\!1,n,\mathbf{u_1},\hdots,\mathbf{u_{K-1}},[\omega_{0j}])+f(m\!+\!1,n,\mathbf{u_1},\hdots,\mathbf{u_{K-1}}),
\]
for some $f$, and so
\begin{equation*}
 F(\zz)=(S_n-I)H(m,n,\mathbf{u_0},\hdots,\mathbf{u_{K-2}},[u_{(K-1)j}])+f(m,n,\mathbf{u_0},\hdots,\mathbf{u_{K-2}}).
\end{equation*}
Adding the trivial CLaw
\begin{equation}
 F_{T}=-(S_n\!-\!I)H, \qquad G_{T}=(S_m\!-\!I)H,\label{Triv}
\end{equation}
to the original densities gives the equivalent densities
\begin{equation*}
 \tilde F = f(m,n,\mathbf{u_0},\hdots,\mathbf{u_{K-2}}), \quad \tilde G=G+(S_m-I)H.
\end{equation*}
The $\tilde G$ density may contain $\Delta$ terms, but these can be removed by adding a trivial density.  Thus the divergence expression for these densities cannot contain any $\Delta$ terms, so in order for it to be a CLaw it must vanish identically and is thus trivial.

The following lemma identifies the kernel of the restricted Euler operator; part of the proof will be used in \S6 to enable us to reconstruct a CLaw from its root.
\begin{Lemma}\label{EulerLem}
 The kernel of $E_{\omega}$ consists of sums of functions that are independent of $\omega$ and its shifts, together with total divergences in the $n$ direction.
\end{Lemma}
\begin{proof}
 In the following, we use the notation $\uu=\{u_{ij}:\ 1\le i \le K-1\}$ and
\begin{equation*}
 \omega_\lambda:=\lambda\omega_{00}+(1-\lambda)g(m,n,\uu),
\end{equation*}
where $g$ is any convenient function (usually, $g=0$). For any differentiable function $f=f(m,n,\uu,[\omega])$,
\begin{align}
 \frac{d}{d \lambda} f(m,n,\uu,[ \omega_\lambda]) =& \sum_{j} \frac{\partial f}{\partial S_n^j \omega_\lambda}S_n^j ( \omega_{00}-g(\nn,\uu))\nonumber\\
=&   ( \omega_{00}\!-\!g)E_{\omega_\lambda}(f(m,n, \uu, [ \omega_\lambda])) +  (S_n-I)h(m,n,\uu,[ \omega],\lambda),\label{eq:genproof1}
\end{align}
for some function $h$. Integrating \eqref{eq:genproof1} with respect to $\lambda$, we obtain
\begin{equation*}
 f(m,n,\uu,[ \omega])=f(m,n,\uu,[g]) +  ( \omega_{00}\!-\!g) \int_{\lambda=0}^{1} E_{\omega_\lambda}(f(m,n,\uu,[\omega_\lambda])) \, d\lambda + (S_n-I)\int_{\lambda=0}^1h(m,n,\uu,[ \omega],\lambda) \, d\lambda
\end{equation*}
If $f\in\text{ker}(E_{\omega})$ then the result follows.
\end{proof}

\section{Using roots to detect equivalence}\label{examplesec1}
For P$\Delta$Es, the fact that a CLaw may involve a large number of points can make it difficult to identify its underlying order. The results of \S3 have established that the root completely characterizes an equivalence class of conservation laws, resolving this difficulty.

As an illustration, consider the transformed dpKdV equation \eqref{dKdVtrans}. Using \textsc{maple}, we have calculated the roots of the CLaws of \eqref{dKdVtrans} that are listed in Table \ref{table1b}. These roots are displayed in Table \ref{table1a}, expressed in terms of the functions $\omega_{ij}$ (because this is more compact than pulling back to write each $\overline{Q}_i$ in terms of $\zz$). The transformed dpKdV equation also has the following CLaw
\begin{small}
\begin{align*}
 F=&\left( -1 \right) ^{m+1} \left( 2\,u_{11} \left( u_{01}+{
\frac {\beta-\alpha}{u_{11}-u_{{12}}}} \right) -\beta \right)   +\left( -1 \right) ^{m+1} \left(\alpha-2 u
_{11} \left( u_{00}+{\frac {\beta-\alpha}{u_{10}-u_{{
11}}}} \right)  \right) 
  ,\\
G=&\frac{2u_{10}\left( -1 \right) ^{m+1}}{y}\left(\left( u_{00}+{\frac {\beta-\alpha}{u_{10}-u_{{1
1}}}} \right) ^{2} - \left( u_{00}+{\frac {\beta-
\alpha}{u_{10}-u_{11}}} \right)  \left( u_{{0
-\!1}}+{\frac {\beta-\alpha}{u_{1-\!1}-u_{10}}} \right)  
  \right) + \\
&+ \frac{\,(-1)^{m+1}}{y}\left(\alpha\left( u_{00}+{\frac {
\beta-\alpha}{u_{10}-u_{11}}} \right) +\alpha \left( u_{
{0-\!1}}+{\frac {\beta-\alpha}{u_{1-\!1}-u_{10}}} \right)-
2\beta\, 
\left( u_{00}+{\frac {\beta-\alpha}{u_{10}-u_{11}}} \right)\right),\\
\text{where}\qquad y=&u_{00}-u_{0-1}-{\frac {\beta-\alpha}{u_{1-\!1}-u_{{
10}}}}+{\frac {\beta-\alpha}{u_{10}-u_{11}}}\,.
\end{align*}
\end{small}

\noindent This apparently high-order CLaw's root is \begin{small}$2\left( -1 \right) ^{m} \left( u_{11}-u_{10}
 \right)$\end{small},  which shows that actually it is equivalent to a multiple of the first CLaw in Table \ref{table1b}.

\begin{small}
\begin{table}[tbp]
\begin{center}
\caption{Roots of the transformed dpKdV}
\begin{tabular}{l}
 \hline
$ \overline{Q}_1=2 \left( -1 \right) ^{m+1}(u_{11}-
u_{10})$
 \T \B\\ \hline 
$\overline{Q}_2=(u_{11}- u_{10})(u_{11}+u_{10}-2 \omega_{00})+\alpha-\beta $
\T \B\\ \hline
$\overline{Q}_3= \left( -1 \right) ^{m} \big((u_{10}- u_{11})(u_{11}+u_{10}+2 \omega_{00}) +\alpha-\beta \big)
 $
\T\B\\ \hline
$ \overline{Q}_4=4 \left( -1 \right) ^{m} \left(\omega_{00}({u_{10}}^{2} -{u_{11}}^{2})+
\alpha u_{11}-\beta\,u_{10}
 \right)
 $\T\B\\ \hline\\[-2ex]
$\overline{Q}_5 =  \left( \beta-\alpha \right)  \left( \omega_{0-1}-
\omega_{00} \right) ^{-2} \left( u_{10}-u_{11}+{\frac {\beta-\alpha}{
\omega_{0-1}-\omega_{00}}} \right) ^{-1}- \left( \beta-\alpha \right)  \left( \omega_{00}-\omega_{01}
 \right) ^{-2} \left( u_{11}-u_{12}+{\frac {\beta-\alpha}{\omega_{00}-
\omega_{01}}} \right) ^{-1}$\\ $\qquad\,\,+ \left( \omega_{00}-\omega_{{01
}} \right) ^{-1}- \left( \omega_{{
0-1}}-\omega_{00} \right) ^{-1} $\T\B\\ \hline \\[-2ex]
$\overline{Q}_6 = \left( \beta-\alpha
 \right)  \left( \omega_{00}-\omega_{01} \right) ^{-2} \left( u_{1
1}-u_{10}+{\frac {\beta-\alpha}{\omega_{00}-\omega_{01}}}
 \right) ^{-1}-
 \left( \beta-\alpha \right)  \left( \omega_{0-1}-\omega_{00}
 \right) ^{-2} \left( u_{10}-u_{1-1}+{\frac {\beta-\alpha}{\omega_{0-1}-
\omega_{00}}} \right) ^{-1}$\\ $\qquad\,\,+ \left( \omega_{0-1}-\omega_{00} \right) ^{-1} - \left( \omega_{00}-\omega_{01} \right) ^{-1}$ \T\B \\ \hline\\[-2ex]
$ \overline{Q}_7=  (\beta-\alpha)(\omega_{00
}-\omega_{01})^{-2}\left\{ \left( u_{11}-u_{12}+{\frac {\beta-
\alpha}{\omega_{00}-\omega_{01}}} \right) ^{-1}+\left( u_{11}-u_{10}+{\frac {\beta-\alpha}{\omega_{00}-\omega_{01}
}} \right) ^{-1}\right\}$\\ $\qquad\,\, + \left( m-n \right)  \overline{Q}_5+n\overline{Q}_6+ 2\left( \omega_{01}-\omega_{00}
 \right) ^{-1}  $ \T\B \\ \hline
\end{tabular}\label{table1a}\end{center}\end{table}
\end{small}

\subsection{The Gardner CLaws for dpKdV}\label{secGardner}
In \cite{RasinSchiff09}, Rasin and Schiff used a discrete version of the Gardner transformation to construct an infinite number of CLaws for the dpKdV equation \eqref{H1}. (Rasin \cite{Rasin2010} has subsequently used the same approach to generate an infinite hierarchy of CLaws for all the equations in the ABS classification and one asymmetric equation.)
They showed that these CLaws were distinct by taking a continuum limit and showing that the resulting CLaws for the continuous equation are distinct.   By using roots, one can prove that their CLaws are distinct without having to take a continuum limit.

In this section, we consider the dpKdV equation on the quad-graph rather than in Kovalevskaya form\footnote{Kovelevskaya form is convenient for proving that the root characterizes each equivalence class of CLaws. However, for any explicit P$\Delta$E, the root (with respect to an appropriate set of initial conditions) can also be calculated without transforming to Kovalevskaya form.}.
The dpKdV equation can be solved for any of the points on the quad-graph; hence we choose the initial conditions $\mathbf{z}=\{m,n,u_{i0},u_{-j1},u_{0k}| i,j,\in \mathbb{N}, k \in \mathbb{Z}\}$ which are shown by dashed lines in Figure \ref{fig8}.

The densities for the CLaws generated by the Gardner transformation are the functions $F_i$ and $G_i$ in the expansion of 
\begin{align}
 F=&-\ln(u_{10}-u_{01})-\ln\left(1+\frac{1}{u_{10}-u_{01}}  \sum_{i=1}^{\infty}v_{00}^{(i)}\epsilon^i\right)=\sum_{i=0}^{\infty}F_i\epsilon^i, \label{H1F}\\
G=& \ln \epsilon-\ln(u_{00}-u_{20})+\ln \left( 1+\frac{1}{v_{00}^{(1)}}\sum_{i=1}^{\infty} v_{00}^{(i+1)}\epsilon^i \right)=\ln{\epsilon}+\sum_{i=0}^{\infty}G_i\epsilon^i,\label{H1G}
\end{align}
in powers of $\epsilon$; here
\begin{align*}
 v_{00}^{(1)}=\frac{1}{u_{00}-u_{20}}, \qquad v_{00}^{(i)}=\frac{1}{u_{00}-u_{20}}\sum_{j=1}^{i-1}v_{00}^{(j)}v_{10}^{(i-j)},
\end{align*}
and $v_{00}^{(i)}$ is referred to as the $i^{\text{th}}$ order $v$ term.
To prove that the CLaws are distinct the following lemma is used.
\begin{Lemma}\label{inflem}
For each $\alpha\in\mathbb{N}$,
\begin{equation}
 \frac{\partial}{\partial u_{(\alpha+1)0}}\, v_{00}^{(\alpha)} \neq 0 \qquad\mathrm{and}\qquad
\frac{\partial}{\partial u_{(\alpha+j)0}}\, v_{00}^{(\alpha)} = 0, \quad j\geq 2. \label{vlem2}
\end{equation}
\end{Lemma}
\begin{proof}
Proof is by induction; the base case ($\alpha=1$) is immediate. Assume that the lemma holds for $\alpha=k-1$, where $k\ge 2$. For each $i\in \mathbb{N}$,
\begin{align*}
 \frac{\partial}{\partial u_{(k+i)0}}\, v_{00}^{(k)}=&\frac{1}{u_{00}-u_{20}}\left( \sum_{j=1}^{k-1} \left(\frac{\partial}{\partial u_{(k+i)0}}\,v_{00}^{(j)}\right)v_{10}^{(k-j)} \! +\!v_{00}^{(j)}\frac{\partial}{\partial u_{(k+i)0}}\left(S_m v_{00}^{(k-j)}\right)\right).
\end{align*}
The first term in the summation is zero for all $j$, because the highest value $j$ can take is $k-1$ but the lowest value of $i$ is $1$.  Similarly, the second term vanishes for all $j\ge 2$, so only one term is left, namely
\begin{align}
\frac{\partial}{\partial u_{(k+i)0}}\, v_{00}^{(k)}=&\left(v_{00}^{(1)}\right)^2S_m\left(\frac{\partial}{\partial u_{(k+i-1)0}}\,v_{00}^{(k-1)}  \right).\label{vproof}
\end{align}
Using the induction hypothesis, if $i=1$ then \eqref{vproof} is nonzero and if $i>1$, \eqref{vproof} is zero.\end{proof}

\begin{figure}[bt]
\begin{center}
 \includegraphics[trim = 0mm 0mm 0mm 0mm, clip,scale=0.622]{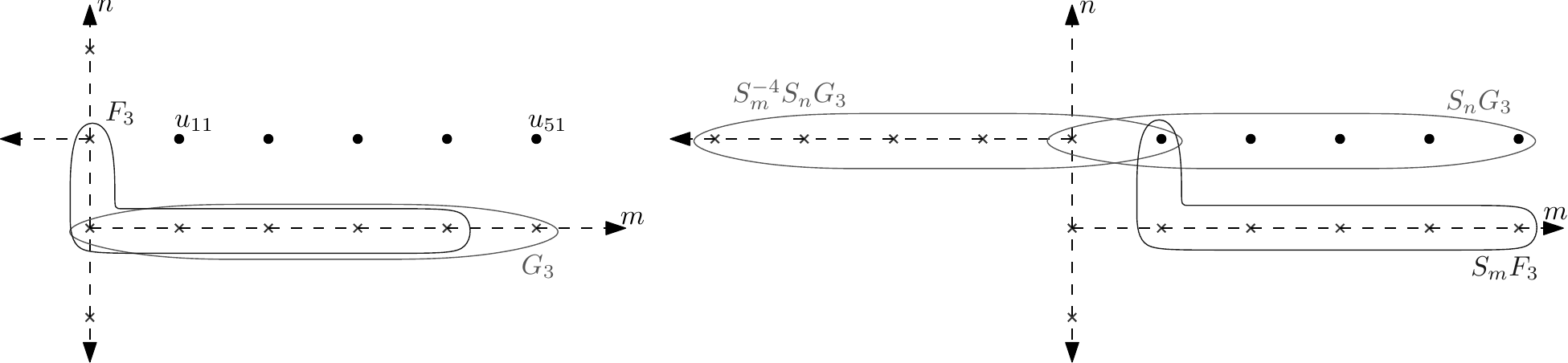}
\caption[A Gardner CLaw of dpKdV and its characteristic ]{The figure on the left shows the densities for the third CLaw in the hierarchy; on the right the extreme shifts of the densities in the characteristic are shown.}\label{fig8}
\end{center}
\end{figure}

A consequence of this lemma is that, for $\alpha\ge 1$
\begin{equation}
 \frac{\partial}{\partial u_{(\alpha+2)0}} \left( \frac{v_{00}^{(\alpha+1)}}{v_{00}^{(1)}} \right)= \frac{1}{v_{00}^{(1)}}\,\frac{\partial}{\partial u_{(\alpha+2)0}} \left( v_{00}^{(\alpha+1)} \right)=v_{00}^{(1)} S_m \left(\frac{\partial}{\partial u_{(\alpha+1)0}}\,v_{00}^{(\alpha)}  \right)  \neq 0.\label{heh}
\end{equation}
As the second factor in \eqref{heh} doesn't depend on $u_{00}$, we obtain
\begin{equation}
\frac{\partial^2}{\partial u_{00}\,\partial u_{(\alpha+2)0}} \left( \frac{v_{00}^{(\alpha+1)}}{v_{00}^{(1)}} \right)\neq 0.\label{muwa1}
\end{equation}
Expanding out \eqref{H1G} shows that, for $\alpha \ge 1$, the highest order $v$ term in $G_\alpha$ is ${v_{00}^{(\alpha+1)}}/{v_{00}^{(1)}}  $; by Lemma \ref{inflem}, this is the only term in $G_\alpha$ to depend on $u_{(\alpha+2)0}$. Therefore $S_n\big( {v_{00}^{(\alpha+1)}}/{v_{00}^{(1)}} \big)  $ is the only term in the CLaw to depend on $u_{(\alpha+2)1}$ and so to depend on $S_m^{\alpha+1}\Delta$ (because $S_mF$ depends only on points in $\zz$ except for $u_{11}$ -- see Figure \ref{fig8}).
Thus the root is 
\begin{equation*}
 E_{\Delta}(S_mF_\alpha + S_nG_\alpha)\big|_{[\Delta]=\mathbf{0}}= \left\{ \frac{\partial S_mF_\alpha}{\partial \Delta}+ \sum_{i=0}^{\alpha} S_m^{-i} \frac{\partial S_nG_\alpha}{\partial S_m^{i}{\Delta}}+ S_m^{-(\alpha+1)}S_n\left(\frac{\partial }{ \partial u_{(\alpha+2)0}}\left( \frac{v_{00}^{(\alpha+1)}}{v_{00}^{(1)}} \right)\right)\right\}\Bigg|_{[\Delta]=\mathbf{0}}.
\end{equation*}
The final term is the only one to depend on $u_{-(\alpha+1),1} $; no other term is shifted as far back. From \eqref{muwa1},
\begin{equation*}
 \frac{\partial}{\partial u_{-(\alpha+1)1}} E_{\Delta}(S_mF\alpha + S_nG_\alpha)|_{[\Delta]=\mathbf{0}}=S_m^{-(\alpha+1)}S_n\left(\frac{\partial^2 }{ \partial u_{00} \partial u_{(\alpha+2)0}}\left( \frac{v_{00}^{(\alpha+1)}}{v_{00}^{(1)}} \right)\right) \neq 0,
\end{equation*}
so this term cannot be a linear combination of the other terms.  Therefore, the characteristic does not vanish on solutions of dpKdV and so the CLaw is nontrivial.  The roots of the lower-order CLaws with densities $(F_1,G_1),\hdots,(F_{\alpha-1},G_{\alpha-1})$ do not depend on $u_{-(\alpha+1)1}  $, so the CLaw $(F_\alpha,G_\alpha)$ cannot be a linear combination of these CLaws and their shifts.  Thus the CLaws generated by the Gardner transformation are distinct.

\section{The converse of Noether's Theorem}\label{secNoet}
One of the most useful significant applications of the root is to establish that the converse of Noether's Theorem holds for difference equations.
Given a Lagrangian, $L$, the Euler--Lagrange equation is
\begin{equation*}
{E}(L):=\sum_{ij} S_m^{-i}S_n^{-j}\left(\frac{\partial L}{\partial u_{ij}} \right)=0.
\end{equation*}
Each symmetry generator for the Euler--Lagrange equation is the prolongation of
\begin{equation}
 X=Q\frac{\partial}{\partial u_{00}}\,;
\end{equation}
the function $Q$ is the \textit{characteristic} of the symmetry generator. 
 In particular, $X$ generates \textit{variational symmetries} if $XL$ is a total divergence, in which case
\begin{equation*}
{E}(XL)\equiv \mathbf{0}.
\end{equation*}
In this case, there exist functions $f$ and $g$ such that
\begin{equation}
 \sum_{ij} \left(S_m^iS_n^j Q\right)\frac{\partial}{\partial u_{ij}}L =  (S_m-I)f + (S_n-I)g.\label{Noet1}
\end{equation}
Summation by parts is then used to rewrite \eqref{Noet1} as
\begin{equation}
 {Q}\cdot{E}(L)= \sum_{i,j} Q S_m^{-i}S_n^{-j}\left(\frac{\partial}{\partial u_{ij}}L \right)=(S_m-I)F+(S_n-I)G,
\end{equation}
for functions $F$ and $G$ whose precise form is irrelevant.
Thus if ${Q}$ is the characteristic of a variational symmetry generator, it is also the characteristic of a CLaw for the Euler--Lagrange equation. Two variational symmetries are equivalent if they differ by a symmetry whose characteristic vanishes on solutions of the Euler--Lagrange equation (i.e.\ a trivial symmetry).  Therefore, if the Euler--Lagrange equation is explicit (in which case the CLaw is trivial if and only only if the root is zero), there is a bijective correspondence between equivalence classes of variational symmetries and CLaws.

\section{Reconstruction of CLaws from roots}\label{reconSec}
If the characteristic of a CLaw is known then the densities for the CLaw can, in principle, be reconstructed using homotopy
operators \cite{MansHydon2004,Olver}.  For PDEs (given an initialization), the root of a CLaw can be calculated by the same approach as we have used; the root is necessarily a characteristic as a consequence of the chain rule. By contrast, the root of a CLaw for a P$\Delta$E may not be a characteristic. So the key step in reconstructing a CLaw from its root is to find a characteristic which has that root.
 
Our starting-point is the proof of Lemma \ref{EulerLem}.  Replacing $f$ by $S_mF$ and using the definition (\ref{euleromega}), we obtain
\begin{equation*}
 F(m+1,n,\uu,[ \omega])= (\omega_{00}-g) \int_{\lambda=0}^{1} \overline{Q}(m,n,\uu,[\omega_\lambda]) \, d\lambda
+ (S_n-I) H(m,n,\uu,[ \omega]) + f(m+1,n,\uu),
\end{equation*}
for some $f$ to be determined.
The $H$ term can be set to zero without loss of generality by adding a trivial CLaw of the second kind, so we can assume that
\begin{equation}
 F(m+1,n,\uu,[ \omega])= (\omega_{00}^{\alpha}-g^\alpha) \int_{\lambda=0}^{1} \overline{Q}(m,n,\uu,[\omega_\lambda]) \, d\lambda + f(m+1,n,\uu).\label{recon1}
\end{equation}
In general, $\overline{Q} $, and as a result \eqref{recon1}, contains negative shifts of $\omega$. These can be removed term-by-term by adding trivial CLaws of the second kind, until one obtains an equivalent density $S_mF$ that has no negative shifts of $\omega$. 

Equation \eqref{recon1} (shifted if necessary, as discussed above) contains all of the $\omega$-dependence of the CLaw. 
Having obtained this, a characteristic for the CLaw is calculated by replacing $\omega_{0j}$ in \eqref{recon1} by $S_n^j\Delta+ \omega_{0j}$. From \eqref{char1},
\begin{equation}
 Q(\zz,[\Delta])=\int_{\lambda=0}^1 E_{\Delta}( F(m+1,n,\uu,[ \Delta + \omega_{00}])) |_{\Delta \mapsto \lambda \Delta} \, d\lambda, \label{eq:charrecon}
\end{equation}
and the CLaw can be written as

\begin{equation*}
 C:=Q(\zz,[\Delta]) \cdot \Delta =  Q\big(\zz, [u_{K0}-\omega(\zz)]\big) \big(u_{K0}-\omega(\zz)\big).
\end{equation*}
Homotopy operators  (see \cite{MansHydon2004})  may then be used to find the densities. Alternatively, once the $\omega$ dependence has been found,
 the arbitrary function $f$ and the other densities can be constructed directly by a variant of the method that we used for O$\Delta$Es (see \cite{HydonCLawsPDiffE} for details). 

\subsection{Example: reconstruction of CLaws of dpKdV}\label{reconsec}

To illustrate this, we will reconstruct two CLaws of dpKdV from their roots.  For both examples, we choose $g(\nn,\uu)=0$, so $\omega_\lambda=\lambda\omega_{00}$. (In practice, we start with this choice and only change it if the integral is singular.)
First, we use the fourth root in Table \ref{table1a}, which gives

\begin{small}
\begin{align}
 F(m+1,n,u_{10},u_{11},\omega_{00})&= \,\omega_{00}\!\int_{\lambda=0}^1\!\!\!\!\! 4 \left( -1 \right) ^{m+1}\! \left\{ (u_{10})^{2}\lambda\,\omega_{00}-\lambda\,\omega_{00}(u_{11})^{2}\!+\alpha u_{11}\!-\!\beta u_{10} \right\} d \lambda + f(m\!+\!1,n,u_{10},u_{11})\nonumber\\
&=\,2 \left( -1 \right) ^{m+1}\! \omega_{00} \left\{ \omega_{00}\left( (u_{10})^{2}-(u_{11})^{2} \right) \!+\!2 \! \left( \alpha u_{11}\!-\!\beta u_{{10}} \right)  \right\} \!+f \left( m+1,n,u_{10},u_{11} \right). \label{reconC4}
\end{align}
\end{small}

\noindent This has no negative shifts of $\omega$ so, from \eqref{recon1} and \eqref{eq:charrecon},  a characteristic is 
\begin{align*}
 Q(\zz,\Delta)=&\int_{\lambda=0}^1 \!\!\left.\frac{\partial}{\partial \Delta}\left\{2 \left( -1 \right) ^{m+1}\! (\Delta+\omega_{00}) \left(( \Delta+\omega_{00})( {u_{10}}^{2}-{u_{11}}^{2} ) \!+\!2 \! \left( \alpha u_{11}\!-\!\beta u_{{10}} \right)  \right)\right\} \right|_{\Delta\mapsto\lambda\Delta} \!\!\!\!\!\! d\lambda,\\
=&2 \left( -1 \right) ^{m+1}\! \left((\Delta+ 2\omega_{00})( {u_{10}}^{2}-{u_{11}}^{2} ) \!+\!2 \! \left( \alpha u_{11}\!-\!\beta u_{{10}} \right)  \right) 
\end{align*}
Thus
\begin{align}
 C:=&Q\Delta=2 \left( -1 \right) ^{m+1} \left( u_{{21}}-\omega_{00}
 \right)\!  \left((u_{21}+ \omega_{00})( {u_{10}}^{2}-{u_{11}}^{2} ) \!+\!2 \! \left( \alpha u_{11}\!-\!\beta u_{{10}} \right)  \right)  \nonumber\\
=&2 \left( -1 \right) ^{m} \!\left(\!\beta\!+\!\alpha\! -u_{{21}}u_{10}\!-\!u_{00}u
_{10}\!-\!u_{{21}}u_{11}\!-\!u_{00}u_{11} \right)  \left( u_{{21}}u_{10}\!-\!u_{{21}}u_{11}\!-\!u_{00}u_{10}\!+\!u_{00}u_{11}\!-\!\beta\!+\!\alpha \right).\label{Cform}
\end{align}
This is a total divergence, so applying the homotopy operator from \cite{MansHydon2004} 
gives the required densities
\begin{align*}
F=&\ \frac{1}{6}\left( -1 \right) ^{m+1} \left(4u_{10}u_{0-1}\beta+ 20 u_{00}\beta\,u_{11}
+4u_{-11}\alpha u_{01}+2u_{-10}u_{01}\beta-10
{u_{00}}^{2}{u_{11}}^{2} \right) \\&+\frac{1}{6}\left( -1 \right) ^{m+1}   \left({u_{00}}^{2}{u_{-1-1}}^{2
}-2 u_{00}\beta u_{-1-1}-4 u_{00}u_{-10}\alpha-10 u
_{00}u_{10}\alpha +2
{u_{-10}}^{2}{u_{00}}^{2}\right)  \\   &+\frac{1}{6}\left( -1 \right) ^{m+1}\left( 7
{u_{01}}^{2}{u_{11}}^{2}-14\alpha u_{11}u_{01}-{u_{{-
10}}}^{2}{u_{01}}^{2}-2{u_{10}}^{2}{u_{0-1}}^{2}-2{
u_{01}}^{2}{u_{-11}}^{2}+5{u_{00}}^{2}{u_{10}}^{2}\right),\\ 
G=&\ \frac{1}{6} \left( -1 \right) ^{m} \left( -10\,u_{10}u_{20}\alpha+
2{u_{-10}}^{2}{u_{00}}^{2}-{u_{10}}^{2}{u_{0-1}}^{2}
-5{u_{00}}^{2}{u_{10}}^{2}\right)\\ &+\frac{1}{6} \left( -1 \right) ^{m}   \left(-2 u_{00}\beta u_{-1-1}+
10u_{00}u_{10}\alpha+{u_{00}}^{2}{u_{-1-1}}^{2}+2 u
_{10}u_{0-1}\beta\right) \\ 
&+\frac{1}{6} \left( -1 \right) ^{m}  \left(4u_{20}u_{1-1}\beta\!+\!5{u_{{
10}}}^{2}{u_{20}}^{2}\!\!-\!2{u_{20}}^{2}{u_{1-1}}^{2}\!\!-\!4
u_{00}u_{-10}\alpha\!-\!12n{\beta}^{2}+12n{\alpha}^{2}
 \right).
\end{align*}
These densities contain points of the form $u_{-1j}$ and $u_{2j}$. To find equivalent densities that are given solely in terms of $\zz$, one must shift the CLaw forwards and then use the dpKdV equation \eqref{dKdVtrans} to pull all terms back onto the initial conditions; this leads to even longer expressions!

In practice it is much easier to use the direct construction method, which leads to more compact expressions for the densities. Starting from \eqref{reconC4}, we know that the densities are of the form
\begin{align*}
F=&2 \left( -1 \right) ^{m} u_{11}\! \left( u_{11} ( {u_{00}}^{2}-{u_{01}}^{2} )\! +2\left( \alpha u_{01}\!-\!\beta u_{00} \right)  \right) \!+\!f \left( m,n,u_{00},u_{01} \right)\!,\\
G=&G(m,n,u_{00},u_{10}).
\end{align*}
Substituting these into the CLaw and evaluating the result on solutions gives 
\begin{align*}
 0=C|_{\Delta=0}=&2 \left( -1 \right) ^{m+1} \left( \!u_{00}\!+\!{\frac {\beta\!\!-\!\!\alpha}{u_{10}\!\!-\!\!u_{11}}}\!\right)\! \left( \left( \!u_{00}\!+\!{\frac {\beta\!\!-\!\!\alpha}{u_{10}\!\!-\!\!u_{11}}}\!\right) ( {u_{10}}^{2}-{u_{11}}^{2} )\! +2\left( \alpha u_{11}\!-\!\beta u_{10} \right)  \right) \\&-2 \left( -1 \right) ^{m} u_{11}\! \left( u_{11} ( {u_{00}}^{2}-{u_{01}}^{2} )\! +2\left( \alpha u_{01}\!-\!\beta u_{00} \right)  \right)+f\left( m+1,n,u_{10},u_{11} \right) \\
 & -f\left( m,n,u_{00},u_{01} \right) +G \left( m,n+1,u_{01},u_{11} \right) -G \left( m,n,u_{00},u_{10} \right).
\end{align*}
Differentiating this expression, we obtain
\begin{equation*}
 0=\frac{\partial^2 C|_{\Delta=0}}{\partial u_{00} \partial u_{01}}=\frac{\partial^2 f}{\partial u_{00} \partial u_{01}}\,.
\end{equation*}
Consequently, $f=\tilde{f}(m,n,u_{00})+h(m,n,u_{01})$; however, $h$ can be set to zero by adding a trivial CLaw.
Therefore
\begin{equation*}
 0=\frac{\partial C|_{\Delta=0}}{ \partial u_{01}} = 4 \left( -1 \right) ^{m}{u_{11}}^{2}u_{01}+4 \left(-1\right)^{m+1}u_{11}\alpha
+\frac{\partial}{\partial u_{01}}  G \left( m,
n+1,u_{01},u_{11} \right), 
\end{equation*}
and so
\begin{equation*}
G \left( m,n,u_{00},u_{10} \right) =2 \left( -1 \right) ^{m+1
}{u_{10}}^{2}{u_{00}}^{2}+4\, \left( -1 \right) ^{m}u_{{10}
}\alpha\,u_{00}+g\left( m,n,u_{10} \right).
\end{equation*}
Dropping the tilde,
\begin{equation*}
 0=\frac{\partial C|_{\Delta=0}}{\partial u_{00}} = -\frac{\partial f}{\partial u_{00}},
\end{equation*}
which yields $f=f(m,n)$.
The final differentiation is
\begin{equation*}
 0=\frac{\partial C|_{\Delta=0}}{\partial u_{10}}=-\frac{\partial g}{\partial u_{10}},
\end{equation*}
so $g=g(m,n)$. The CLaw now simplifies to
\begin{equation*}
 (S_m-I)f(m,n) +(S_n-I)g(m,n) =2 \left( -1 \right) ^{m}({\alpha}^{2}-{\beta}^{2}),
\end{equation*}
a solution of which is $f=0$ and $g=2n(-1)^m(\alpha^2-\beta^2)$.  So the reconstructed densities are
\begin{align*}
 F=& 2 u_{11} \left( -1 \right) ^{m} \left( u_{11}({u_{{00}}}^{2}-{u_{01}}^{2})+2\alpha\,u_{01}-2\beta
\,u_{00} \right) ,\\
G=&2 \left( -1 \right) ^{m} \left( -{u_{10}}^{2}{u_{00}}^{2}+2
\,u_{10}\alpha\,u_{00}+n({\alpha}^{2}-{\beta}^{2}) \right),
\end{align*}
which are equivalent to the densities found by the homotopy method, as they have the same root.

For a more complicated example, consider the densities of the sixth Claw in Table \ref{table1b}.
Its root depends on $\omega_{0-1}$, $\omega_{00}$ and $\omega_{01}$.  The characteristic is constructed so that terms depending on $\omega_{0-1}$ do not depend on $\omega_{01}$. The term that depends on $\omega_{0-1}$ is
\begin{equation*}
h_1(m,n,\mathbf{u_1},\omega_{0-1},\omega_{00}):= \left( \omega_{0-1}\!-\!\omega_{00} \right) ^{-1}\!-\!
 \left( \beta\!-\!\alpha \right)\!  \left( \omega_{0-1}\!-\!\omega_{00}
 \right) ^{-2}\! \left(\! u_{10}+{\frac {\beta\!-\!\alpha}{\omega_{0-1}\!-\!
\omega_{00}}}-u_{1-1}\! \right) ^{-1}\!\!\!\!.
\end{equation*}
 Then let
\begin{align*}
h_2(m,n,\mathbf{u_1},\omega_{00},\omega_{01}):=&\, \omega_{00} \overline{Q}_6(m,n,\mathbf{u_1},[\lambda\omega]) + (S_n-I)\left(\omega_{00} h_1 (\nn,\mathbf{u_1},[\lambda\omega]) \right)\\=&-{\frac {\omega_{00}u_{10}-\omega_{01}u_{10}-\omega_{{00}
}u_{11}+\omega_{01}u_{11}}{-\lambda\,\omega_{00}u_{11}+
\lambda\,\omega_{00}u_{10}-\lambda\,\omega_{01}u_{10}-
\beta+\lambda\,\omega_{01}u_{11}+\alpha}}\,,
\end{align*}
which is a function that only depends on $\omega_{00}$ and $ \omega_{01}$; this is integrated with respect to $\lambda$ to obtain
\begin{align}
S_mF|_{[\Delta]=0}=&\int h_2\, \mathrm{d} \lambda + f(m+1,n,u_{10},u_{11},u_{12})\nonumber\\
=&-\ln  \left( {\frac {(\omega_{01}-\omega_{00})(u_{11}-u_{10})+\alpha-\beta}{
\alpha-\beta}} \right) + f(m+1,n,u_{10},u_{11},u_{12}).\label{TheAnswer}
\end{align}
As in the last example, we could calculate the characteristic from \eqref{TheAnswer} and use the homotopy operator to find densities for the CLaw.  Once again, however,  the direct construction method is preferable.
By shifting \eqref{TheAnswer} and choosing $G$ to depend on appropriate values of $\zz$, the densities have the form
\begin{align*}
F=& -\!\ln\!  \left( {\frac {(u_{12}-u_{11})(u_{01}-u_{00})+\alpha-\beta}{\alpha-\beta}} \right)\! +\!f
 \left( m,n,u_{00},u_{01},u_{{02}} \right)\!,\\
G=&G \left( m
,n,u_{00},u_{10},u_{01},u_{11} \right).
\end{align*}
 The dependence on the dpKdV equation is already determined by \eqref{TheAnswer}, so all that remains is to find $f$ and $G$. The same process (differential elimination followed by integration) is used as before. Skipping the details, we obtain
\begin{equation*}
 F= -\!\ln\!  \left( {\frac {(u_{12}-u_{11})(u_{01}-u_{00})+\alpha-\beta}{\alpha-\beta}} \right),\qquad G= -\ln  \left( u_{10}-u_{11} \right),
\end{equation*}
as required.

\section{Finding CLaws}\label{examplesec}
\subsection{The Adjoint of the Linearized Symmetry Operator}\label{adjsec}
The G\^{a}teaux derivative of a functional $P$ is the operator
 defined in \cite{MansHydon2004} by
\begin{align*}
 \DD_P(Q)=&\lim_{ \epsilon \to 0 } \left( \frac{P\left[u+\epsilon Q[u]\right] -
P[u]}{\epsilon} \right)=\left\{\frac{d}{d \epsilon} P\left[u+\epsilon Q[u]\right]\right\}\Big|_{\epsilon =
0}.
\end{align*}
Explicitly, the G\^{a}teaux derivative of $P$ is the
shift operator with entries 
\begin{equation*}
  \DD_P=\sum_{ij}\frac{\partial P}{\partial u_{ij}}\,
S_m^iS_n^j.
\end{equation*}
Therefore its adjoint with respect to the $\ell_2$ inner product is 
\begin{equation*}
  \DD^*_P=\sum_{i,j}\left(S_m^{-i}S_n^{-j}\frac{\partial P}{\partial
u_{ij}}\right) S_m^{-i}S_n^{-j},
\end{equation*}
and so
the Euler operator, $E$, is defined by 
\begin{equation*}
E(P[u])= \DD^{*}_{P}(1),
\end{equation*}
just as for PDEs. Using the Leibniz rule,
\begin{equation}
 E(P\cdot Q)= D^{*}_{P\cdot Q} (1) = D^{*}_{P} (Q) + D^{*}_{ Q} (P).
\end{equation}
The action  of the vector field $X= Q\, {\partial}/{\partial
u_{00}}$ on the functional $P$ is
\begin{equation*}
 \mbox{pr } X (P) = \DD_P(Q).
\end{equation*}
So the linearized
symmetry condition for a given difference equation, $\Delta={0}$, is
\begin{equation*}
 0=\left.\DD_{\Delta}(Q)\right|_{[\Delta]={0}}.
\end{equation*}
The Euler operator acting on a expression is zero if and only if that expression
is a total divergence \cite{MansHydon2004}.  Therefore 
 $Q$ is a 
characteristic of a CLaw if and only if
\begin{equation*}
 0=E(Q\cdot \Delta)= \DD^{*}_{\Delta}(Q) + \DD^{*}_{Q}(\Delta).
\end{equation*}
Restricting this to solutions of the difference equation gives a necessary
condition for $Q$ to be a characteristic:
\begin{equation}
 0=\left.\DD^{*}_{\Delta}(Q)\right|_{[\Delta]={0}}.\label{AdjResult}
\end{equation}
In other words, the characteristics are members of the kernel of the adjoint of
the linearized symmetry condition (ALSC), restricted to solutions.
Arriola \cite{Arriola} showed that \eqref{AdjResult} must be satisfied for first integrals of autonomous ordinary difference equations.   In a notable paper that introduces the idea of co-recursion operators for integrable difference equations \cite{XenitidisCoSym10}, Mikhailov \textit{et al.} define a \textit{cosymmetry} of a difference equation as being a member of the kernel of the ALSC.  They state (parenthetically) that cosymmetries are characteristics of CLaws, but do not justify this.
\begin{table}[t]
\begin{center}
\caption{Solutions of the ALSC for the potential Lotka--Volterra equation}\label{tablepLVq}
\begin{tabular}{l}
\hline
 $\overline{Q}_1={\frac { \left( -n+2 \right) u_{10}}{u_{0-1}u_{00}}}+{\frac {n-1}{u_{00}}}+{\frac {nu_{02}u_{10}}{{u_{00}}^{2} \left( u_{10}+u_{01} \right) }}
 $\T\B\\
$\overline{Q}_2= \frac{1}{u_{00}}+{\frac {u_{02}u_{10}}{{u_{00}}^{2}
 \left( u_{10}+u_{01} \right) }}-{\frac {u_{10}}{u_{0-1}u_{00}}}
  $\T\B\\
$\overline{Q}_3=-{\frac {mu_{10}u_{-10}}{  u_{01}u_{00}\left( u_{10}+u_{-10} \right)}}
-{\frac { \left( m+1 \right) u_{20}}{  u_{00}\left( u_{20}+u_{00} \right) \left( u_{10}+u_{01} \right) 
}}   $\T\B\\
$\overline{Q}_4= -{\frac {u_{10}u_{-10}}{ u_{01}u_{00} \left( u_{10}+u_{-10} \right)}}-{\frac {u_{10}u_{20}}{ u_{00}\left( u_{20}+u_{00} \right)  \left( u_{10}+u_{01} \right) }}
  $\T\B\\
$\overline{Q}_5= {\frac { \left( -1 \right) ^{m+n} \left( {u_{10}}^{2}u_{01}+{u_{00}}^{2}(u_{10}+u_{01}) \right) }{{u_{00}}^{2}u_{01} \left( u_{10}+u_{01} \right) }}
 $\T\B\\
$\overline{Q}_6= {\frac {{u_{10}}^{2}u_{01}-{u_{00}}^{2}(u_{10}+u_{01})}{{u_{00}}^{2} u_{01}\left( u_{10}+u_{01}
 \right) }}
  $\T\B\\
 \hline 
\end{tabular}\end{center}\end{table}

In this section, we give an example of a P$\Delta$E where not every cosymmetry is a characteristic of a CLaw (see \cite{AncoBlumanPartI,AncoBlumanPartII} for a discussion of this point for differential equations).  First, we find functions $Q$ that satisfy \eqref{AdjResult}, using methods similar to those used to find symmetries of difference equations.  Then we find additional constraints on $Q$ by applying the difference Euler operator to $Q\Delta$.  Finally, we reconstruct the densities using homotopy operators or inspection. For brevity, we omit most details of the calculations.

\subsection{CLaws of the Potential Lotka-Volterra equation}\label{sec:findingroots}

The potential Lotka--Volterra equation (pLV)
\begin{equation}
 \frac{u_{11}}{u_{00}}-\frac{u_{01}}{u_{10}}=1. \label{eq:pLV}
\end{equation}
is an integrable equation on the quad-graph. (It belongs to Class 4 of Hietarinta \& Viallet's classification of quadratic quad-graph equations with polynomial degree growth \cite{hv}, and is a potential form of the discrete Lotka-Volterra equation introduced by Hirota \& Tsujimoto \cite{ht}.) 

Rasin \& Hydon's method for finding symmetries of quad-graph equations \cite{RasinSym2007} is readily adapted to find solutions of the ALSC, shifted (for convenience) to
\begin{equation}
0=\left\{\left(\overline{Q}-S_n\left(\omega_{,2}\overline{Q}\right)-S_m\left(\omega_{,3}\overline{Q} \right) - S_mS_n\left(\omega_{,1} \overline{Q}  \right)\right)\right\}\big|_{[\Delta]=0}. \label{ALSCcon}
\end{equation}
We search for solutions of \eqref{ALSCcon} which are pulled back onto the initial conditions $\zz=\{m,n,u_{i0},u_{0j}\}$.
In particular, we will look for roots of the form
\begin{equation}
 \overline{Q} = \overline{Q}(m,n,u_{-10},u_{0-1},u_{00},u_{10},u_{01},u_{20},u_{02}).\label{eq:rootguess}
\end{equation}
The corresponding CLaws depend on $m,n,u_{-10},u_{0-1},u_{00},u_{10}$ and $u_{01}$ only, and are therefore called `five-point CLaws'.

By differential elimination and integration, one obtains the solutions of the ALSC; these are listed in Table \ref{tablepLVq}. However, when one tries to use these roots to construct the corresponding CLaws of the pLV equation, the algorithm fails for $\overline{Q}_1$ (see Table \ref{CLawspLV}). It turns out that $E(Q \Delta)\neq0$, and therefore this cosymmetry does not correspond to the characteristic of a CLaw. This is not surprising, as \eqref{AdjResult} is necessary but not sufficient; nevertheless, to the best of our knowledge, this is the first example of such a cosymmetry of a P$\Delta$E. It raises  an interesting question: can an additional constraint be found that guarantees a solution of the ALSC is a root of a CLaw without the need to work through the (lengthy) process of reconstructing the characteristic? For instance, Mikhailov \textit{et al.} \cite{XenitidisCoSym10} have used a co-recursion operator to generate an infinite hierarchy of cosymmetries for the Viallet equation, so a simple test to show that these produce an infinite hierarchy of CLaws would be useful.

\begin{small}
\begin{table}[t]
\begin{center}
\caption{Five-point CLaws of the potential Lotka--Volterra equation}\label{CLawspLV}
\begin{tabular}{l}
 \hline\\[-2ex]
The solution $\overline{Q}_1$ of the ALSC is not the root of a characteristic.\\
\hline\\[-2ex]
$F_2={\frac { \left(u_{-\!10}- u_{01} \right)  \left( u_{0-\!1}+u_{-\!10} \right) }{u_{0-\!1}u_{-10}}},\qquad G_2= -{\frac {u_{-\!10}}{u_{0-\!1}}}
  $\\[1ex]
\hline\\[-2ex]
$F_3= m\ln  \left( {\frac {u_{01}}{u_{-\!10}}} \right) -\ln  \left( {u_{00}} \right),\qquad G_3=\left( m+1 \right) \ln  \left( {\frac {u_{10}+u_{-\!10}}{u_{10}}} \right) 
  -\ln   \left( u_{10}+u_{-\!10}
 \right) 
  $\\[1ex]
\hline\\[-2ex]
$F_4= \ln  \left( {\frac {u_{01}}{u_{-\!10}}} \right),\qquad G_4= \ln \left( {\frac {  u_{10}+u_{-\!10} }{u_{10}}}\right) 
  $\\[1ex]
\hline\\[-2ex]
$F_5= {\frac { \left( -1 \right) ^{m+n}u_{00}}{u_{01}}},\qquad G_5= {\frac { \left( -1 \right) ^{m+n} \left( {u_{00}}^{2}-{u_{10}}^{2} \right) }{u_{00}u_{10}}}
  $\\[1ex]
\hline\\[-2ex]
$F_6= {\frac {2\,u_{01}-u_{00}}{u_{01}}},\qquad G_6= -{\frac {{u_{00}}^{2}+{u_{10}}^{2}}{u_{00}u_{10}}}  $\\[1ex]
\hline
\end{tabular}\end{center}\end{table}
\end{small}

\section{Acknowledgements}
We thank the Natural Environment Research Council for funding this research. We also thank the referees for their very helpful recommendations.

\end{document}